\documentclass[11pt,english,a4paper,leqno]{smfart}
\usepackage[english]{babel}
\usepackage{amssymb,calrsfs,bbold}
\usepackage[dvips]{graphicx}
\usepackage{enumerate}
\usepackage{amsmath}
\usepackage{stmaryrd}

\newtheorem{theorem}{Theorem}[section]
\newtheorem{lemma}[theorem]{Lemma}

\newtheorem{corollary}[theorem]{Corollary}

\newtheorem*{coro*}{Corollary}

{\theoremstyle{definition}}

{\theoremstyle{definition}}

\theoremstyle{definition}

\newtheorem{question}[theorem]{Question}
\newtheorem{fact}[theorem]{Fact}

\newtheorem*{fact*}{Fact}

\newtheorem*{claim*}{\rm Claim}

{\theoremstyle{definition}}

\def\D{\ensuremath{\mathbb D}}

\def\T{\ensuremath{\mathbb T}}

\def\Z{\ensuremath{\mathbb Z}}

\def\C{\ensuremath{\mathbb C}}

\def\N{\ensuremath{\mathbb N}}

\DeclareMathOperator{\dens}{dens}

\newcommand{\hy}{hypercyclic}
\newcommand{\fhy}{frequently hypercyclic}
\newcommand{\ufhy}{$\mathcal{U}$-frequently hypercyclic}
\newcommand{\ops}{operators}
\newcommand{\op}{operator}

\newcommand{\erg}{ergodic}

\newcommand{\pr}{probability}

\newcommand{\pss}[2]{\ensuremath{{\langle #1,#2\rangle}}}

\newcommand{\ba}[1]{\overline{#1}}

\newcommand{\wh}[1]{\widehat{#1}}
\usepackage[all]{xy}
\entrymodifiers={+!!<0pt,\fontdimen22\textfont2>}
\def\apl#1#2#3{\,#1\!:\!\!\xymatrix@C=17pt{#2\ar[r]&#3}
}
\def\aplba#1#2{\xymatrix@C=20pt{#1\!\ar@{|->}[r]&\!#2}
}

\newcommand{\di}{\underline{\textrm{dens}}}
\newcommand{\ds}{\overline{\textrm{dens}}}
\newcommand{\bx}{\mathcal{B}(X)}
\newcommand{\nt}[2]{\mathcal{N}_{\,T}(#1,#2)}
\renewcommand{\t}{$T$}
\newcommand{\hc}{HC($T$)}
\newcommand{\fhc}{FHC($T$)}
\newcommand{\ufhc}{UFHC($T$)}

\begin{document}

\date{\today}

\title[Irregularly visiting orbits]{Frequently hypercyclic operators with irregularly visiting orbits}

\author{S. Grivaux}
\address{CNRS, Laboratoire Paul Painlev\'e, UMR 8524\\
Universit\'{e} de Lille\\
Cit\'e Scientifique, B\^atiment M2\\
59655 Villeneuve d'Ascq Cedex\\
France}
\email{sophie.grivaux@math.univ-lille1.fr}

\thanks{This work was supported in part by EU IRSES grant AOS (PIRSES-GA-2012-
318910), by the Labex CEMPI (ANR-11-LABX-0007-01), and by the grant ANR-17-CE40-0021 of the French
National Research Agency ANR (project Front).}



\begin{abstract}
We prove that a bounded operator $T$ on a separable Banach space $X$ satisfying a strong form of the Frequent Hypercyclicity Criterion (which implies in particular that the operator is universal in the sense of Glasner and Weiss) admits frequently hypercyclic vectors with \emph{irregularly visiting orbits}, i.e. vectors $x\in X$ such that the set  $\nt{x}{U}=\{n\ge 1\,;\,T^{n}x\in U\}$ of return times of $x$ into $U$ under the action of $T$ has positive lower density for every non-empty open set $U\subseteq X$, but there exists a non-empty open set $U_0\subseteq X$ such that $\nt{x}{U_0}$ has no density.
\end{abstract}

\maketitle

\section{Introduction}\label{Intro}
Let $X$ be a separable infinite-dimensional Banach space, and let $\mathcal{B}(X)$ be the space of bounded linear \ops\ on $X$. We are interested in this paper in the study of dynamics of certain \ops\ $T\in\bx$, and the existence of vectors whose iterates under the action of \t\ have an ``irregular'' behavior. Given $T\in\bx$, $x\in X$, and $U\subseteq X$ a non-empty open set, we denote by $\nt{x}{U}=\{n\ge 1\,;\,T^{n}x\in U\}$ the set of return times of $x$ into $U$ under the action of $T$. The \op\ \t\  is said to be \emph{\hy}\ when there exists a vector $x\in X$ with dense orbit under the action of \t, i.e.\ when $\nt{x}{U}$ is non-empty for every non-empty open set $U\subseteq X$, and \emph{\fhy}\ (resp. \emph{\ufhy}) when there exists $x\in X$ such that $\di\,\nt{x}{U}>0$ (resp. $\ds\,\nt{x}{U}>0$) for every non-empty open set $U\subseteq X$. We denote respectively by $\di\,A$, $\ds\,A$ and $\textrm{dens}\,A$  the lower density, the upper density, and if it exists, the density of a subset $A$ of $\N$:
\begin{align*}
 \di\,A&=\liminf_{N\to+\infty}\dfrac{1}{N}\,\#\bigl ([1,N]\cap A \bigr)\\
 \ds\,A&=\limsup_{N\to+\infty}\dfrac{1}{N}\,\#\bigl ([1,N]\cap A \bigr)\\
 \textrm{dens}\,A&=\lim_{N\to+\infty}\dfrac{1}{N}\,\#\bigl ([1,N]\cap A \bigr).
\end{align*}
Vectors $x$ with one of the properties above are called respectively \hy, \fhy, and \ufhy\ vectors. We denote by \hc, \fhc, and \ufhc\  these three sets of vectors. When \t\ is \hy, \hc\ is a dense $G_{\delta }$ subset of $X$. The set \fhc, although dense in $X$, is meager in $X$ for every \fhy\ \op\ \t\ (\cite{Moo}, \cite{BR}, see also \cite{GM}), while \ufhc\ is comeager in $X$ for every \ufhy\ \op\ \t\ (\cite{BR}).
The notion of frequent hypercyclicity was introduced in the paper \cite{BG1}, while $\mathcal{U}$-frequent hypercyclicity was first considered by Shkarin in \cite{Sh}. These two concepts have been the object of an important amount of work in recent years.
 We refer the reader to the books \cite{GEP} and \cite{BM}, as well as to the papers \cite{GM}, \cite{BMPP}, \cite{BoGE2017}, \cite{Me} and \cite{GMM} (among many others) for an in-depth study of frequent and $\mathcal{U}$-frequent hypercyclicity and related phenomena.
\par\medskip
There are (as of now) essentially two known ways of constructing frequently \hy\ vectors for an \op\ $T\in\bx$: by an explicit construction, or by using ergodic theory. 
In the first approach, explicit \fhy\ vectors are constructed as a series of vectors whose iterates have suitable properties. The most classical construction of his type is the one yielding the so-called Frequent Hypercyclicity Criterion, first proved in \cite{BG1} and then generalized in \cite{BoGE}. Another explicit construction of such vectors, making use of some assumptions concerning the periodic points of the operator, is given in \cite{GMM}.
In particular, operators with the so-called Operator Specification Property, which were shown in \cite{BMP} to be frequently hypercyclic, satisfy this criterion.
 The second, widely used, approach for proving that a given \op\ is \fhy, is to use ergodic theory: one shows that the \op\ is \emph{ergodic} in the sense that it admits an invariant \pr\ measure $m$ with full (topological) support with respect to which it defines an ergodic transformation of the space. An application of Birkhoff's pointwise ergodic theorem then shows that \t\ is \fhy. More precisely, if $T\in\bx$ is \erg\ with respect to a \pr\ measure $m$ on $X$ such that $m(U)>0$ for every non-empty open set $U\subseteq X$, then, for every such $U\subseteq X$, $\textrm{dens}\,\nt{x}{U}=m(U)$ for $m$-almost every $x\in X$. It follows by considering a countable basis $(U_{p})_{p\ge 1}$ of non-empty open subsets of $X$ that $m$-almost every vector $x\in X$ is \fhy\ for \t, and satisfies
$\textrm{dens}\,\nt{x}{U_{p}}=m(U_{p})$ for every $p\ge 1$. Hence the following question, which was posed in a first version, dating from 2013, of the survey paper \cite{GTen}, arises naturally:

\begin{question}\label{Question 1}
Does there exist a \fhy\ \op\ \t\ on a Banach space $X$ which admits a \fhy\ vector $x\in X$ such that for some non-empty open set $U\subseteq X$, the set $\nt{x}{U}$ has no density, i.e.\ $\di\,\nt{x}{U}<\ds\,\nt{x}{U}$? 
\end{question}

A related question is due to Shkarin, who asked in \cite{Sh} whether all frequently \hy\ \op s $T\in\bx$ admit a frequently \hy\ vector $x$ such that, for every  non-empty open set $U\subseteq X$, $\nt{x}{U}$ contains a set of positive density. As mentioned above, all ergodic \op s satisfy this property.
\par\medskip
Observe that Question \ref{Question 1} is very easy to answer if one withdraws the requirement that the vector $x\in X$ be frequently \hy\ for \t. Indeed, if $T\in\bx$ is any \ufhy\ \op, \ufhc\ is comeager in $X$ while \fhc\ is meager. Hence \ufhc\,$\setminus$\,\fhc\ is comeager in $X$, and any vector $x$ belonging to this set has the property that for some non-empty open set $U\subseteq X$, $\di\,\nt{x}{U}=0$ while $\ds\,\nt{x}{U}>0$.
\par\medskip
A first progress concerning Question \ref{Question 1} was made by Y. Puig de Dios in \cite{Puig}. He proved there the following result: for any \fhy\ \op\ \t\ on a Banach space $X$, any \fhy\ vector $x\in X$ and any non-empty open subset $U$ of $X$ with the property that none of the sets $\bigcup_{n=0}^N T^{-n}U$, $N\ge 0$, is  dense in $X$, the set $\nt{x}{U}$ has different lower density and upper Banach density.
\par\medskip
Our aim in this note is to answer Question \ref{Question 1} in the affirmative by showing that \op s $T\in\bx$ satisfying a particular form of the Frequent Hypercyclicity Criterion admit frequently \hy\ vectors $x$ with \emph{irregularly visiting orbits}, i.e.\ such that for some non-empty open set $U\subseteq X$, $\di\,\nt{x}{U}<\ds\,\nt{x}{U}$. Operators of this kind were considered in \cite{G}, where it was shown that they are universal in the sense of Glasner and Weiss \cite{GW}, i.e.\ represent in a certain sense all ergodic \pr\ preserving systems. We say that a vector $x_{0}\in X$ is \emph{bicyclic} for an \op\ $T\in\bx$ if there exist vectors $x_{n}\in X$, $n\in\Z\setminus \{0\}$, such that $Tx_{n}=x_{n+1}$ for every $n\in\Z$ and the linear span of the vectors $x_{n}$, $n\in\Z$, is dense in $X$. When $x_{0}$ is a bicyclic vector for \t, we write such associated vectors $x_{n}$ as $x_{n}=T^{n}x_{0}$, $n\in\Z$, even when $T$ is not invertible.
\par\medskip
Here is the main result of this note:
\begin{theorem}\label{Theorem 1}
 Let \t\ be a bounded \op\ on a separable Banach space $X$ which admits a vector $x_{0}\in X$ with the  following three properties:
 \begin{enumerate}
 \item[\emph{(a)}] $x_{0}$ is a bicyclic vector for $T$;
  \item[\emph{(b)}] the series $\sum_{n\in\Z}T^{-n}x_{0}$ is unconditionally convergent;
  \item[\emph{(c)}] there exists a non-zero functional $x_{0}^{*}\in X^{*}$ and a finite subset $F\subseteq \Z$ such that $\pss{x_{0}^{*}}{T^{n}x_{0}}=0$ for every $n\in\Z\setminus F$.
 \end{enumerate}
 Then \t\ admits a frequently \hy\ vector $x$ with an irregularly visiting orbit. More precisely,
 $\underline{\emph{dens}}\,\nt{x}{U_{0}}<\overline{\emph{dens}}\,\nt{x}{U_{0}}$, where $U_{0}=\{y\in X\,;\,\Re e\pss{x_{0}^{*}}{y}>0\}$.
\end{theorem}
As already shown in \cite{G}, many classical frequently \hy\ \op s satisfy the assumptions of Theorem \ref{Theorem 1}. Among them are multiples $\omega B$ with $|\omega |>1$ of the unilateral backward shift on $\ell_{p}(\N)$, $1\le p<+\infty$, or $c_{0}(\N)$, as well as adjoints $M^{*}_{\varphi }$ of multipliers of the Hardy space $H^{2}(\D)$ when
$\smash{\apl{\varphi }{\D}{\C}}$ is a non-constant holomorphic function on the open unit disk $\D=\{\lambda \in\C\,;\,|\lambda |<1\}$ such that $\T\subseteq\varphi (\D)$, where 
$\T=\{\lambda \in\C\,;\,|\lambda |=1\}$ denotes the unit circle. This relies on the fact that these \op s admit a spanning unimodular eigenvector field which is sufficiently smooth. Recall that if $\sigma $ is a \pr\ measure on $\T$, 
$E\in L^{2}(\T,\sigma ;X)$ is a unimodular eigenvector field of $T\in \bx$ (with respect to $\sigma $) if $TE(\lambda )=\lambda E(\lambda )$ $\sigma $-a.e\ on $\T$. We will only consider here the case where $\sigma $ is the normalized Lebesgue measure $d\lambda$ on $\T$. The Fourier coefficients of $E$ are defined by \[\wh{E}(n)=\displaystyle\int_{\T}\lambda ^{-n}E(\lambda )\,d\lambda, \quad n\in\Z.\]
\begin{corollary}\label{Corollary 2}
 Let \t\ be a bounded \op\ on a separable complex Banach space $X$ which admits a unimodular eigenvector field $E$ with the following properties:
 \begin{enumerate}
  \item [\emph{(a)}] for every measurable subset $B$ of $\T$ of full Lebesgue measure, 
  $\overline{\emph{span}\vphantom{t}}\,[E(\lambda )\,;\,\lambda \in B]=X$;
  \item[\emph{(b)}] the series $\sum_{n\in\Z}\wh{E}(n)$ is unconditionally convergent;
  \item[\emph{(c)}] there exists a trigonometric polynomial $p$ and a non-zero functional $x_{0}^{*}\in X^{*}$ such that $\pss{x_{0}^{*}}{E(\lambda )}=p(\lambda )$ a.e.\ on $\T$.
 \end{enumerate}
 Then \t\ admits a frequently \hy\ vector with an irregularly visiting orbit.
 \end{corollary}
The proof of Theorem \ref{Theorem 1} relies on a construction inspired from the proof of the Frequent Hypercyclicity Criterion of \cite{BG1} or \cite{BoGE}  of a particular family $(D_{s})_{s\ge 1}$ of subsets of $\N$ with positive lower density. It is carried out in Section \ref{Section 2}. That \op s satisfying the assumptions of Corollary \ref{Corollary 2} satisfy those of Theorem \ref{Theorem 1} is already proved in \cite{G}, so we do not give the argument here. We also refer the reader to \cite{G} for more details concerning examples of \op s to which Theorem \ref{Theorem 1} and Corollary \ref{Corollary 2} apply, and which thus have frequently \hy\ vectors with irregularly visiting orbits. We collect in Section \ref{Section 3} some further remarks and open questions.
\par\smallskip
Given a family $(A_i)_{i\in I}$ of disjoint subsets of $\Z$, we denote their (disjoint) union by \[\coprod_{i\in I}A_i.\]

\section{Proof of Theorem \ref{Theorem 1}}\label{Section 2}
Before starting the proof, let us observe that \op s satisfying the assumptions of Theorem \ref{Theorem 1} satisfy the Frequent Hypercyclicity Criterion. Indeed, the set of vectors of the form $z=\sum_{k\in G}a_{k}T^{-k}x_{0}$, $G$ a finite subset of $\Z$, $a_{k}\in \C$ for every $k\in G$, is dense in $X$ by assumption (a), and by assumption (b) any vector $z$ of this form is such that the series 
$\sum_{n\in\Z}T^{n}z$ is unconditionally convergent. So we know already that such  \op s must be \fhy. The \fhy\ vectors with irregularly visiting orbit which we will construct will be defined as  unconditionally convergent series $x=\sum_{p\ge 1}z_{p}$, where $z_{p}=\sum_{n\in G_{p}}a_{n}T^{-n}x_{0}$ and the finite sets $G_{p}$, $p\ge 1$, are successive disjoint intervals in $\N$.
\par\medskip
For every integer $s\ge 1$, let $\mathcal{G}_{s}$ denote the class of (finite or infinite) subsets $G$ of $\Z$ such that $\min_{n\in G}|n|\ge 2^{s}$. We also denote by $\varepsilon _{s}$ the smallest constant with the property that 
\[
\Bigl |\Bigl |\,\sum_{n\in G}\beta _{n}T^{-n}x_{0}\, \Bigr|  \Bigr|\le \varepsilon _{s}\,\max_{n\in G}|\beta _{n}| 
\]
for every $G\in\mathcal{G}_{s}$ and every family $(\beta _{n})_{n\in G}$ of scalars.
Since the series $\sum_{n\in\Z}T^{-n}x_{0}$ is unconditionally convergent, $\varepsilon _{s}$ tends to $0$ as $s$ tends to infinity. Let then $(c_{s})_{s\ge 1}$ be a sequence of nonnegative numbers with the  following three properties:
\begin{enumerate}
 \item [($\alpha$)] $\limsup\limits _{s\to+\infty}c_{s}=+\infty$;
 \item[($\beta$)] $\varepsilon _{s} \displaystyle\sum_{1\le s'<s}c _{s'}\longrightarrow 0 $ as $s\longrightarrow +\infty$;
 \item[($\gamma$)] the series $\displaystyle\sum_{s\ge 1}c_{s}\varepsilon _{s}$ is convergent.
\end{enumerate}
Such a sequence does exist: it suffices to consider a (fast increasing) sequence $(s_{j})_{j\ge 1}$ of integers such that the series $\sum_{j\ge 1}j^{2}\varepsilon _{s_{j}}$ is convergent, and to define the sequence $(c_{s})_{s\ge 1}$ by setting $c_{s}=0$ if $s\not\in\{s_{j}\,;\,j\ge 1\}$ and $c_{s_{j}}=j$ for every $j\ge 1$.
\par\medskip
Let then $(x_{s})_{s\ge 1}$ be a dense sequence of vectors of $X$ of the form
\begin{equation}\label{eqenplus}
 x_{s}=\sum_{|j|\le 2^{s}}a_{j}^{(s)}T^{j}x_{0},\quad \textrm{with}\ \max_{|j|\le 2^{s}}|a_{j}^{(s)}|\le c_{s}\;\textrm{ for every } s\ge 1.
\end{equation}
Such a sequence exists by assumptions (a) and ($\alpha$).
Indeed, since $x_0$ is a bicyclic vector for $T$, there exists a dense sequence $(y_{q})_{q\ge 1}$  of vectors of $X$ of the form
\[
y_{q}=\sum_{|j|\le N_q}b_{j}^{(q)}T^{j}x_{0},\quad \textrm{with }N_q\ge 1 \textrm{ and }b_{j}^{(q)}\in\C \textrm{ for every }|j|\le N_q.
\]
By assumption ($\alpha$), there exists a strictly increasing sequence $(s_{q})_{q\ge 1}$ of integers such that $N_q\le 2^{s_q}$ and $\max_{|j|\le N_q}|b_{j}^{(q)}|\le c_{s_q}$ for every $q\ge 1$. Set $x_{s_q}=y_q$ for every $q\ge 1$, and $x_s=0$ for every $s\in\N\setminus\{s_q\;;\;q\ge 1\}$. The sequence $(x_{s})_{s\ge 1}$ is then dense in 
$X$ and satisfies (\ref{eqenplus}).

\par\smallskip
 Our \fhy\ vector with an irregularly visiting orbit will be of the form
\begin{equation}\label{Equation 1}
 x=\sum_{s\ge 1}\sum_{k\in D_{s}}T^{-k}x_{s}
\end{equation}
where the sets $D_{s}$, $s\ge 1$, are subsets of $\N$ with positive lower density which are sufficiently separated from each other. 

\begin{lemma}\label{Lemma 1}
 Suppose that $(D_{s})_{s\ge 1}$ is a sequence of subsets of $\N$ with the following properties:
 \begin{enumerate}
  \item [\emph{(i)}]$\min D_{s}\ge 2^{s+1}$ for every $s\ge 1$;
  \item[\emph{(ii)}] for every $s\ge 1$ and every $i,i'\in D_s$ with $i\not =i'$, $|i-i'|\ge 2^{s+1}+1$;
  \item[\emph{(iii)}] the sets $D_{s}$, $s\ge 1$, are pairwise disjoint and, more precisely, for every $s, s'\ge 1$ with $s\neq s'$, every $i\in D_{s}$ and every $i'\in D_{s'}$, $|i-i'|\ge 2^{\max(s,s')+1}+1$;
  \item[\emph{(iv)}] for every $s\ge 1$, $D_{s}$ has positive lower density.
 \end{enumerate}
Then the vector $x$ given by (\ref{Equation 1}) is well-defined and \fhy\ for \t. 
\end{lemma}
\begin{proof}
 The proof of this lemma is very similar to that of the Frequent Hypercyclicity Criterion in \cite{BG1} or \cite{BoGE}, so we will be somewhat sketchy. We have
 \begin{equation}
x=\sum_{s\ge 1}\sum_{k\in D_{s}}\sum_{|j|\le 2^{s}}a_{j}^{(s)}T^{-k+j}x_{0}
=\sum_{s\ge 1}\sum_{k\in D_{s}}\sum_{i\in I_{k,s}}b_{i}T^{-i}x_{0}\notag
\end{equation}
where $I_{k,s}=[k-2^{s},k+2^{s}]$ and $b_{i}=a_{k-i}^{(s)}$ for every $i\in I_{k,s}$, $k\in D_{s}$, $s\ge 1$. Conditions (ii) and (iii) imply that the intervals $I_{k,s}$ are pairwise disjoint. By condition (i), we have $\min I_{k,s}\ge 2^{s}$ for every $k\in D_{s}$ and $s\ge 1$, so that the set \[I_s=\coprod_{k\in D_s}I_{k,s}\] belongs to $\mathcal{G}_{s}$. It follows that the series $\sum_{i\in I_{s}}b_{i}T^{-i}x_{0}$, which is unconditionally convergent since $\sup_{i\in I_s}|b_i|=\max_{|j|\le 2^s}|a_j^{(s)}|\le c_s$, satisfies
\[
\Bigl| \Bigl|\sum_{i\in I_{s}}b_{i}T^{-i}x_{0} \Bigr|  \Bigr|\le\max_{i\in I_{s}}|b_{i}|\,.\,\varepsilon _{s}\le c_{s}\varepsilon _{s}. 
\]
Hence the series $\sum_{s\ge 1}\sum_{i\in I_{s}}b_{i}T^{-i}x_{0}$ is (unconditionally) convergent by  ($\gamma$), and the vector $x$ is well-defined.
\par\medskip
Fix now $r\ge 1$. For any $n\in D_{r}$, we have
\begin{equation}\label{Equation 3}
 T^{n}x=\sum_{\genfrac{}{}{0pt}{1}{s\ge 1}{s\neq r}}\sum_{k\in D_{s}}T^{n-k}x_{s}+x_{r}+\sum_{k\in D_{r}\setminus \{n\}}T^{n-k}x_{r}.
\end{equation}
For every $k\in D_{r}\setminus \{n\}$
\[
||T^{n-k}x_{r}||=\Bigl |\Bigl |\sum_{i\in[k-n-2^{r},\,k-n+2^{r}]}a_{k-n-i}^{(r)}T^{-i}x_{0} \Bigr|  \Bigr|. 
\]
By (ii), $|k-n|\ge 2^{r+1}+1$, so that the set \[\coprod_{k\in D_{r}\setminus \{n\}}[k-n-2^{r},k-n+2^{r}]\] belongs to the class $\mathcal{G}_{r}$. It follows that \[\Bigl |\Bigl |\sum_{k\in D_{r}\setminus \{n\}}T^{n-k}x_{r}\Bigr |\Bigr |\le c_{r}\varepsilon _{r}.\] 
\par\medskip
As to the first term in the expression (\ref{Equation 3}), we estimate it in a similar way, using this time (iii). For every $s\neq r$ and every $k\in D_{s}$,
\[
T^{n-k}x_{s}=\sum_{i\in[k-n-2^{s},\,k-n+2^{s}]}a_{k-n-i}^{(s)}T^{-i}x_{0},
\]
and $|k-n|\ge 2^{\max(r,s)+1}+1$ by (iii). It follows that the set \[\coprod_{k\in D_s}[k-n-2^{s},k-n+2^{s}]\] 
belongs to the class $\mathcal{G}_{r}$ if $s<r$, and to the class $\mathcal{G}_{s}$ if $s>r$. Hence \[\Bigl |\Bigl |\sum_{k\in D_s}T^{n-k}x_{s}\Bigr |\Bigr |\le c_{s}\varepsilon _{r}\quad \textrm{if } s<r \quad \textrm{and} \quad\Bigl |\Bigl |\sum_{k\in D_s}T^{n-k}x_{s}\Bigr |\Bigr |\le c_{s}\varepsilon _{s}\quad \textrm{if } s>r.\] Thus 
\begin{align*}
 \Bigl | \Bigl |\sum_{\genfrac{}{}{0pt}{1}{s\ge 1}{s\neq r}}\sum_{k\in D_{s}}T^{n-k}x_{s} \Bigr| \Bigr|&\le\Bigl (\sum_{s<r} c_{s} \Bigr)\varepsilon _{r}+
\sum_{s>r}c_{s}\varepsilon _{s},
\intertext{so that}
||T^{n}x-x_{r}||&\le\Bigl (\sum_{s<r} c_{s} \Bigr)\varepsilon _{r}+
\sum_{s\ge r}c_{s}\varepsilon _{s}.
\end{align*}
Conditions ($\beta$) and ($\gamma$) on the sequence $(c_{s})_{s\ge 1}$ then imply that given $\varepsilon >0$, there exists $r_{0}\ge 1$ such that $\sup_{n\in D_{r}}||T^{n}x-x_{r}||<\varepsilon $ for every $r\ge r_{0}$. Thus $x$ is a \fhy\ vector for \t.
\end{proof}

Under additional assumptions on the sets $D_{s}$, the vector $x$ given by 
(\ref{Equation 1}) has an irregularly visiting orbit. In the statement of Lemma \ref{Lemma 2} below, we denote by $d(n,A)$ the distance of an integer $n\in\Z$ to a subset $A$ of $\Z$. Let also $d\ge 1$ be an integer such that $F\subseteq [-d,d]$. 
% For every $|k|\le d$, we set $\alpha _{k}=\pss{x_{0}^{*}}{T^{k}x_{0}}$.
\begin{lemma}\label{Lemma 2}
 Suppose that $(D_{s})_{s\ge 1}$ is a family of subsets of $\N$ satisfying assumptions \emph{(i)} and \emph{(iv)} of Lemma \ref{Lemma 1}, as well as the following reinforced versions \emph{(ii')} and \emph{(iii')} of \emph{(ii)} and \emph{(iii)} respectively:\
 \begin{enumerate}
\item [\emph{(ii')}] for every $s\ge 1$ and every $i,i'\in D_s$ with $i\not =i'$, $|i-i'|\ge 2^{s+1}+2d+1$;
  \item [\emph{(iii')}] for every $s,s'\ge 1$ with $s\neq s'$, and every $i\in D_{s}$, $i'\in D_{s'}$, $|i-i'|\ge 2^{\max(s,s')+1}+2d+1$;
 \end{enumerate}
and
\begin{enumerate}
 \item [\emph{(v)}] there exists a strictly increasing sequence $(N_{l})_{l\ge 1}$ of integers such that
\smallskip
 \begin{enumerate}
 \item[\emph{(v-a)}] $d(N_{l},D_{s})\ge 2^{s}+d$ for every $s\ge 1$ and every $l\ge 1$;
 \item[\emph{(v-b)}] the series $\sum_{s\ge 1} 2^s \alpha_s$ is convergent, where $\alpha_s=\sup_{l\ge 1} \frac{1}{N_l}\# D_s\cap[1,N_l]$ for every  $s\ge 1$;
 \item[\emph{(v-c)}] there exist two subsequences $(N_l^{(1)})_{l\ge 1}$ and $(N_l^{(2)})_{l\ge 1}$ of $(N_l)_{l\ge 1}$ and, for every $s\ge 1$, two positive constants $0<\beta_s<\gamma_s$
 such that  \[
\dfrac{1}{N_{l}^{(1)}}\# D_{s}\cap[1,N_{l}^{(1)}] \longrightarrow \beta_s \quad \textrm{and}\quad \dfrac{1}{N_{l}^{(2)}}\# D_{s}\cap[1,N_{l}^{(2)}] \longrightarrow \gamma_s             
            \]
 as  $l$ tends to infinity.
\end{enumerate}
\end{enumerate}
Then the vector $x\in X$ given by the expression \emph{(\ref{Equation 1})} has an irregularly visiting orbit. More precisely, the set $D=\{n\ge 1\,;\,\Re e\pss{x_{0}^{*}}{T^{n}x}>0\}$ has no density.
\end{lemma}
\begin{proof}
 Using the notation introduced in the proof of Lemma \ref{Lemma 1}, we write $x$ as
 \[
x=\sum_{s\ge 1}\sum_{k\in D_{s}}\sum_{i\in I_{k,s}}b_{i}T^{-i}x_{0}
\]
where $I_{k,s}=[k-2^{s},k+2^{s}]$ and $b_{i}=a_{k-i}^{(s)}$ for every $i\in I_{k,s}$, $k\in D_{s}$, $s\ge 1$. By assumptions (ii') and (iii'), the intervals $I_{k,s}$ are pairwise disjoint, and the distance between two of them is always at least $2d+1$. For every $n\in \Z$, we have
\[
T^{n}x=\sum_{s\ge 1}\sum_{k\in D_{s}}\sum_{i\in I_{k,s}}a_{k-i}^{(s)}T^{n-i}x_{0}.
\]
Hence $\pss{x_{0}^{*}}{T^{n}x}$ can be non-zero only when $n$ belongs to the set 
\[
\coprod_{s\ge 1}\,\coprod_{k\in D_{s}} I_{k,s}+[-d,d].
\]
When $n\in I_{k,s}+[-d,d]$, $\pss{x_{0}^{*}}{T^{n}x}=\pss{x_{0}^{*}}{T^{n-k}x_s}$.
If we define, for every $s\ge 1$ and every $k\in D_{s}$, 
$
I_{k,s}^{+}=\bigl\{ n\in I_{k,s}+[-d,d]\,;\, \Re e\pss{x_{0}^{*}}{T^{n}x}>0\bigr\},
$ we have
\[
D=\coprod_{s\ge 1}\coprod_{k\in D_{s}} I_{k,s}^{+}
\]
 and
$
\#I_{k,s}^{+}= 
\#\bigl\{ p\in [-(2^{s}+d),2^{s}+d\,]\,;\, \Re e\pss{x_{0}^{*}}{T^{p}x_s}>0 \bigr\}.
$
The quantity $\#I_{k,s}^{+}$ does not depend on $k$, and we denote it by $r_{s}$. Since the sequence $(x_{s})_{s\ge 1}$ is dense in $X$, there exists $s_0\ge 1$ such that $r_{s_0}\ge 1$.  
Moreover, for every $l\ge 1$, we have by assumption (v-a) that if $k\in D_s\cap [1,N_{l}]$, $I_{k,s}^{+}\subseteq I_{k,s}+[-d,d]
\subseteq [1,N_{l}]$, so that
\[
D\cap[1,N_{l}]=\coprod_{s\ge 1}\,\,\coprod_{k\in D_{s}\cap [1,N_{l}]}I_{k,s}^{+}.
\]
Hence $\#D\cap [1,N_{l}]=\displaystyle\sum_{s\ge 1}r_{s}\,\,\# D_s\cap [1,N_{l}]$. 
Since $r_{s}\le 2^{s+1}+2d+1$, we have 
\[
r_{s}\,\dfrac{1}{N_l}\,\# D_s\cap [1,N_{l}]\le ({2^{s+1}+2d+1})\, \alpha_s        \quad \textrm{for every } l\ge 1.                       
\] 
It follows then from assumptions (v-b) and (v-c) and the dominated convergence theorem that 
\[
\liminf\limits _{l\to+\infty}\,\,\dfrac{1}{N_{l}}\,\,\# D\cap [1,N_{l}]\le\sum_{s\ge 1} r_s\beta _{s}\quad \textrm{while}\quad \limsup\limits_{l\to+\infty}\,\,\dfrac{1}{N_{l}}\,\,\# D\cap [1,N_{l}]\ge \sum_{s\ge 1} r_s\gamma _{s}.
\]
Since $\beta _{s}<\gamma _{s}$ for every $s\ge 1$ and since $r_{s_0}\ge 1$, we deduce that the set $D$ has no density, which proves our claim.
\end{proof}

In order to conclude the proof of Theorem \ref{Theorem 1}, it remains to construct sets $D_{s}$ satisfying the assumptions (i), (ii'), (iii'), (iv) and (v). We set, for every $j\ge 1$,  $I_{j}=[2^{j},2^{j+1})$. The $I_{j}$'s are successive intervals of respective lengths $2^{j}$, which cover the interval $[2,+\infty)$. Let $p\ge 1$ be such that $2^{s+1+p}\ge 2^{s+1}+2d+1$ for every $s\ge 1$. We define for each $s\ge 1$  sequences $(I_{j}^{(s)})_{j\ge s}$ of intervals of $\N$ by setting 
\[I_{j}^{(s)}=[2^{j}+2^{j-1}+\ldots+2^{j-s+1}, 2^{j}+2^{j-1}+\ldots+2^{j-s})\]
for every $s\ge 1$  and every $j\ge s$. All the intervals
$I_{j}^{(s)}$, $s\ge 1$, $j\ge s$, are pairwise disjoint.
We have $\#I_{j}^{(s)}=2^{j-s}$, so that $\#I_{j}^{(s)}\ge 2^{s+2+p}$ for every $j\ge 2s+p+2$. We now set, for every $j\ge 2s+p+2$
\[
L_{j}^{(s)}=\bigl \{ i\in I_{j}^{(s)}\,;\,d(i,\N\setminus I_{j}^{(s)})\ge 2^{s+1+p} \quad \textrm{and}\quad i\equiv 0\,[2^{s+1+p}]\,\bigr\} ,
\]
and define \[
\Delta _{s}=\coprod_{j\ge 2s+p+2}L_{j}^{(s)}.            
           \]
These sets $\Delta_s$ are pairwise disjoint.
We have $\min \Delta _{s}\ge 2^{2s+p+2}\ge 2^{s+1}$. Also, suppose that $i$ and $i'$ belong to $\Delta _{s}$ and $\Delta _{s'}$ respectively, with $i\neq i'$. Let $j\ge 2s+p+2$ and $j'\ge 2s+p+2$ be the unique indices such that $i\in L_{j}^{(s)}$ and $i'\in L_{j'}^{(s')}$.
\par\medskip
-- if $s=s'$, then either $j=j'$, in which case $|i-i'|\ge 2^{s+1+p}$, or $j\neq j'$, in which case $i\in I_{j}^{(s)}$ and $i'\in I_{j'}^{(s)}\subseteq \N\setminus I_{j}^{(s)}$. By the definition of $L_{j}^{(s)}$, $|i-i'|\ge 2^{s+1+p}$. In both cases, $|i-i'|\ge 2^{s+1+p}\ge 2^{s+1}+2d+1$.
\par\medskip
-- if $s\neq s'$, we have $i\in I_{j}^{(s)}$ and $i'\in I_{j'}^{(s)}\subseteq \N\setminus I_{j}^{(s)}$, so that $|i-i'|\ge 2^{s+1+p}$; also $i'\in I_{j'}^{(s')}$ and $i\in \N\setminus I_{j'}^{(s')}$, so that $|i-i'|\ge 2^{s'+1+p}$. Hence 
\[
|i-i'|\ge 2^{\max(s,s')+1+p}\ge 2^{\max(s,s')+1}+2d+1.
\]
\par\medskip
Hence the family $(\Delta _{s})_{s\ge 1}$ satisfies assumptions (i), (ii') and (iii').
In order to obtain a family satisfying additionally assumptions (iv) and (v), we consider the sets $D_{s}$ defined by setting
\[
D_{s}=\coprod_{\genfrac{}{}{0pt}{1}{j\ge 2s+p+2}{j\in J}}L_{j}^{(s)}\quad\textrm{for every }s\ge 1, \textrm{ where } J=5\N\cup(5\N+2).            
           \]
 Since $D_{s}\subseteq \Delta _{s}$ for every $s\ge 1$, the family $(D_s)_{s\ge 1}$ clearly satisfies assumptions (i), (ii') and (iii').

\par\smallskip
We will now need the following technical fact, which concerns the sums \[S(a,b)=2^{-b}\,\displaystyle\sum_{\genfrac{}{}{0pt}{1}{a<j\le b}{j\in J}}2^{j},\quad\textrm{where } 0\le a<b \textrm{ are integers}.\] 

\begin{fact}\label{Fact 0}
For every $a\ge 0$, we have
\begin{enumerate}
 \item [$\bullet$] $\sup\limits_{b>a}S(a,b)\le\dfrac{64}{31}\cdot$
 \item[$\bullet$] 
$
\lim\limits_{\genfrac{}{}{0pt}{1}{b\to+\infty}{b\in B}}S(a,b)=
\begin{cases}
36/31&\textrm{when}\ B= 5\N\\
18/31&\textrm{when}\ B= 5\N+1\\
40/31&\textrm{when}\ B= 5\N+2\\
20/31&\textrm{when}\ B= 5\N+3\\
10/31&\textrm{when}\ B= 5\N+4.
\end{cases}
$
\end{enumerate}
\end{fact}

\begin{proof}
 These observations rely on the fact that 
 \[
S(a,b)=2^{-b}\cdot\dfrac{32}{31}\cdot\Bigl [\Bigl (2^{\,5\lfloor \frac{b}{5}\rfloor}-2^{\,5\lfloor \frac{a}{5}\rfloor} \Bigr) +4\Bigl (2^{\,5\lfloor \frac{b-2}{5}\rfloor}-2^{\,5\lfloor \frac{a-2}{5}\rfloor} \Bigr) \Bigr].
\]
If we denote, for every integer $c\ge 0$, by $r(c)\in\{0,1,2,3,4\}$ the residue class of $c$ modulo $5$, we have $2^{\,5\lfloor \frac{c}{5}\rfloor}=2^{\,c-r(c)}$, so that
\[
S(a,b)\sim\dfrac{32}{31}\,\Bigl ( 2^{\,-r(b)}+4\,.\,2^{\,-2-r(b-2)}\Bigr)=
\dfrac{32}{31}\,\Bigl (2^{\,-r(b)}+2^{\,-r(b-2)} \Bigr)\quad\textrm{ as } b\rightarrow +\infty. 
\]
Hence
\[
\lim\limits_{\genfrac{}{}{0pt}{1}{b\to+\infty}{b\in B}}S(a,b)=
\begin{cases}
 \dfrac{32}{31}\Bigl (1+\dfrac{1}{8} \Bigr) =\dfrac{36}{31}&\textrm{when}\ B= 5\N\\[2ex]
\dfrac{32}{31}\Bigl (\dfrac{1}{2}+\dfrac{1}{16} \Bigr) =\dfrac{18}{31}&\textrm{when}\ B= 5\N+1\\[2ex]
\dfrac{32}{31}\Bigl (\dfrac{1}{4}+1\Bigr) =\dfrac{40}{31}&\textrm{when}\ B= 5\N+2\\[2ex]
\dfrac{32}{31}\Bigl (\dfrac{1}{8}+\dfrac{1}{2} \Bigr) =\dfrac{20}{31}&\textrm{when}\ B= 5\N+3\\[2ex]
\dfrac{32}{31}\Bigl (\dfrac{1}{16}+\dfrac{1}{4} \Bigr) =\dfrac{10}{31}&\textrm{when}\ B= 5\N+4.
\end{cases}
\]
\end{proof}
For every $n\ge 2$, let $j_{n}=\lfloor\log_{2}n\rfloor$: $2^{j_n}\le n<2^{j_n+1}$, and $j_n$ is the unique integer $j\ge 1$ such that $n$ belongs to $I_{j}$. We have 
\begin{equation}\label{eq4}
 \#\, D_{s}\cap[1,2^{j_{n}})\le\#\, D_{s}\cap[1,n)\le\#\, D_{s}\cap[1,2^{j_{n}+1})
 \end{equation}
 and
 \[
 \#\,D_{s}\cap[1,2^{m})=\#\coprod_{\genfrac{}{}{0pt}{1}{2s+p+2\le j<m}{j\in J}}L_{j}^{(s)}=
 \sum_{\genfrac{}{}{0pt}{1}{2s+p+2\le j<m}{j\in J}}\#\,L_{j}^{(s)} \quad\textrm{ for every }m\ge 1.
 \]
 Now
$
 2^{\,-(s+1+p)}\#\,I_{j}^{(s)}-2\le\#\, L_{j}^{(s)}\le 2^{\,-(s+1+p)}\#\,I_{j}^{(s)},$ i.e.
\[2^{\,j-2s-p-1}-2\le\#\,L_{j}^{(s)}\le 2^{\,j-2s-p-1}.\] 
Hence
\begin{align*}
\#\,D_{s}\cap[1,2^{m})&\ge\sum _{\genfrac{}{}{0pt}{1}{2s+p+2\le j<m}{j\in J}}\bigl (2^{\,j-2s-p-1} -2\bigr)\\&=2^{\,-2s-p-1}\,2^{\,m-1}S(2s+p+3,m-1)-2(m-2s-p-2).
\end{align*}
It follows from Fact \ref{Fact 0} that 
\[
\liminf_{m\to+\infty}\,2^{-m}\,\#\,D_{s}\cap [1,2^{m})\ge \dfrac{10}{31}\cdot 2^{-2s-p-2}
\]
so that by (\ref{eq4}) $\di\, D_{s}\,\ge\dfrac{10}{31}\cdot 2^{-2s-p-3}$ for every $s\ge 1$. Each set $D_{s}$ has positive lower density, and assumption (iv) holds true. Also, Fact \ref{Fact 0} implies that 
\[
2^{\,-m}\,\# \,D_{s}\cap[1,2^{m}]\le\dfrac{64}{31}\cdot 2^{\,-2s-p-2}\quad \textrm{for every }m\ge 1,
\]
 so that 
\[
\dfrac{1}{n}\,\#\,D_{s}\cap[1,n]\le \dfrac{64}{31}\cdot 2^{\,-2s-p-1}\quad \textrm{for every }n\ge 1.
\] Assumption (v-b) is hence satisfied, whatever the choice of the sequence $(N_{l})_{l\ge 1}$.
\par\medskip
In order to define a suitable sequence $(N_{l})_{l\ge 1}$, we enumerate the set
$J\cap[p+4,+\infty)$ as an increasing sequence $(q_{l})_{l\ge 1}$, and set $N_{l}=2^{\,q_{l}+1}$ for every $l\ge 1$.
Since the distance between two elements of $J$ is always at least $2$, $q_l+1$ never belongs to $J$.
In order to estimate, for every $s\ge 1$ and every $l\ge 1$, the distance $d(N_{l},D_{s})$ of $N_l$ to $D_s$, we consider two cases:
\par\smallskip
-- if $q_l\ge \min J\cap [2s+p+2,+\infty)$, 
\begin{eqnarray*}
d(N_{l},D_{s})&\ge&\min\bigl (2^{\,q_l+2}-2^{\,q_l+1},2^{\,q_l+1}-(2^{\,q_l}+2^{\,q_l-1}+\ldots+2^{\,q_l-s}) \bigr)\\ 
&\ge& \min\bigl (2^{\,q_l+1}, 2^{\,q_l-s}\bigr)\ge 2^{\,s+p+2}\ge 2^{\,s}+d; 
\end{eqnarray*}
\par\smallskip
-- if $q_l< \min J\cap [2s+p+2,+\infty)$,  then $q_l\le \min J\cap [2s+p+2,+\infty)-2$, and
\begin{eqnarray*}
d(N_{l},D_{s})&=& 2^{\,\min J\cap [2s+p+2,+\infty)}-2^{\,q_l+1} \ge 2^{\,2s+p+1}\ge 2^{\,s}+d.
\end{eqnarray*}
Hence assumption (v-a) is satisfied.
Also,
\begin{align*}
\#\,D_{s}\cap[1,N_{l}]
=\sum_{\genfrac{}{}{0pt}{1}{2s+p+2\le j\le q_{l}}{j\in J}}\#\,L_{j}^{(s)}
\end{align*}
so that $\dfrac{1}{N_{l}}\,\#\, D_{s}\cap[1,N_{l}]\sim 2^{\,-2s-p-2}\,S( 2s+p+3,q_{l} )$ as $l\rightarrow +\infty$. It follows from Fact \ref{Fact 0} that for every $s\ge 1$,
\begin{align*}
\lim_{\genfrac{}{}{0pt}{1}{l\to+\infty}{q_{l}\in 5\N}}\dfrac{1}{N_{l}}\,\#\, D_{s}\cap[1,N_{l}]&=2^{\,-2s-p-2}\cdot\dfrac{36}{31}
\intertext{and}
\lim_{\genfrac{}{}{0pt}{1}{l\to+\infty}{q_{l}\in 5\N+2}}\dfrac{1}{N_{l}}\,\#\, D_{s}\cap[1,N_{l}]&=2^{\,-2s-p-2}\cdot\dfrac{40}{31}\cdot
\end{align*}
These two limits being distinct, (v-c) is satisfied. So the family of sets $(D_{s})_{s\ge 1}$ satisfies (i), (ii'), (iii'), (iv) and (v), and Theorem \ref{Theorem 1} follows from Lemmas \ref{Lemma 1} and \ref{Lemma 2}.

\section{Further remarks and open questions}\label{Section 3}
The \op s satisfying the assumptions of Theorem \ref{Theorem 1} admit a \fhy\ vector $x\in X$ such that $\nt{x}{U}$ has no density for a particular choice of the set $U$. A natural question is whether it is possible to enlarge the class of open sets $U$ for which this irregular behavior holds. Observe that it cannot hold for \emph{all} non-empty open sets: if $U$ is any \t-invariant dense open subset of $X$, $\nt{x}{U}$ is a cofinite set for any hypercyclic vector $x$, and thus $\dens\nt{x}{U}=1$. But it makes sense to ask:
\begin{question}\label{Question 2}
 Does there exist a bounded \op\ \t\ on a separable Banach space $X$ which admits a \fhy\ vector $x\in X$ such that $\nt{x}{U}$ has no density for every bounded non-empty open subset $U\subseteq X$? Or any non-empty open ball $U$? Or any other natural class of non-empty open subsets $U$ of $X$ (for instance, the class considered in \cite{Puig} of non-empty open sets $U$ with the property that none of the sets $\bigcup_{n=0}^N T^{-n}U$, $N\ge 0$, is  dense in $X$)? Do \op s satisfying the assumptions of Theorem \ref{Theorem 1} admit \fhy\ vectors with an irregularly visiting orbit in one of these strong senses?
\end{question}
It seems likely that the existence of \fhy\ vectors with an irregularly visiting orbit is a much more general phenomenon that what is seen here, and that actually \emph{all} \fhy\ \op s admit such a vector.
\begin{question}\label{Question 3}
 Does every \fhy\ \op\ $T\in\mathcal{B}(X)$ admit a \fhy\ vector with an irregularly visiting orbit?
\end{question}
Observe that if $T\in\mathcal{B}(X)$ is an ergodic \op\ with respect to a \pr\ measure $m$ on $X$ with full support, and $(U_{p})_{p\ge 1}$ is any collection of non-empty open sets in $X$, $\dens\nt{x}{U_{p}}=m(U_{p})$ for $m$-a.e.\  $x\in X$. If $(U_{p})_{p\ge 1}$ is a countable basis of open sets in $X$ which is stable by taking finite unions, such vectors $x$ are easily seen to be such that $\dens\nt{x}{U}=m(U)$ for every non-empty open set $U\subseteq X$ such that $m(\partial U)=0$, where $\partial U$ denotes the boundary $\ba{U}\setminus U$ of $U$. Whether this property can hold for \emph{all} non-empty open sets $U$ does not seem clear at all.
\par\medskip
Lastly, we recall a question from \cite{Sh}, already mentioned in the introduction:
\begin{question}\label{Question 5}
 Does there exist an \op\ $T\in\bx$ admitting a \fhy\ vector $x\in X$ such that, for some non-empty open subset $U\subseteq X$, $\nt{x}{U}$ contains no subset $A$ of $\N$ admitting a positive density?
 \end{question}

\end{document}